%% file: BSCknots.tex
\documentclass{amsart}
\usepackage{amsthm}
\usepackage{amsmath}
\usepackage{amssymb}
\usepackage{enumerate}
\usepackage{graphicx}
\usepackage[hidelinks,pagebackref,pdftex]{hyperref}
\usepackage{booktabs}
\usepackage{color}
\usepackage{import}

\AtBeginDocument{%
   \def\MR#1{}
}

\usepackage{marginnote}
\long\def\@savemarbox#1#2{\global\setbox#1\vtop{\hsize\marginparwidth 
    \@parboxrestore\tiny\raggedright #2}}
\renewcommand*{\backref}[1]{}
\renewcommand*{\backrefalt}[4]{
  \ifcase #1
  [No citations.]
  \or [#2]
  \else [#2]
  \fi }

\numberwithin{equation}{section}
\theoremstyle{plain}
\newtheorem{theorem}{Theorem}
\numberwithin{theorem}{section}

\newtheorem{lemma}[theorem]{Lemma}

\newtheorem{proposition}[theorem]{Proposition}

\newtheorem*{namedtheorem}{\theoremname}
\newcommand{\theoremname}{testing}

\theoremstyle{definition}

\newcommand{\from}{\colon} 

\newcommand{\half}{{\frac{1}{2}}}

\newcommand{\HH}{{\mathbb{H}}}
\newcommand{\RR}{{\mathbb{R}}}
\newcommand{\ZZ}{{\mathbb{Z}}}

\newcommand{\CC}{{\mathbb{C}}}

\newcommand{\SSS}{{\mathbb{S}}}

\newcommand{\calL}{{\mathcal{L}}}

\newcommand{\refthm}[1]{Theorem~\ref{Thm:#1}}
\newcommand{\reflem}[1]{Lemma~\ref{Lem:#1}}
\newcommand{\refprop}[1]{Proposition~\ref{Prop:#1}}

\newcommand{\refsec}[1]{Section~\ref{Sec:#1}}
\newcommand{\reffig}[1]{Figure~\ref{Fig:#1}}

\newcommand{\bdy}{\partial}

\newcommand{\vol}{\operatorname{vol}}
\newcommand{\PSL}{\operatorname{PSL}}

\newcommand{\area}{\operatorname{area}}
\newcommand{\len}{\operatorname{length}}
\newcommand{\inj}{\operatorname{inj}}

\newcommand{\id}{{\mathop{id}}}
\newcommand{\core}{{\operatorname{core}}}
\newcommand{\PMod}{{\operatorname{PMod}}}

\title{Spacious knots}

\author{Autumn E. Kent}
\address{Department of Mathematics, University of Wisconsin, 480 Lincoln Drive, Madison, WI 53706, USA}

\author{Jessica S.~Purcell}
\address{School of Mathematical Sciences, 9 Rainforest Walk, Monash University, VIC 3800, Australia}

\begin{document}

\begin{abstract}
  We show that there exist hyperbolic knots in the 3-sphere such that the set of points of large injectivity radius in the complement take up the bulk of the volume. More precisely, given a finite volume hyperbolic manifold, for any bound $R>0$ on injectivity radius, consider the set of points with injectivity radius at least $R$; we call this the $R$-thick part of the manifold. We show that for any $\epsilon>0$, there exists a knot $K$ in the 3-sphere so that the ratio of the volume of the $R$-thick part of the knot complement to the volume of the knot complement is at least $1-\epsilon$. As $R$ approaches infinity, and as $\epsilon$ approaches $0$, this gives a sequence of knots that is said to Benjamini--Schramm converge to hyperbolic space. This answers a question of Brock and Dunfield.
\end{abstract}

\maketitle

\section{Introduction}\label{Sec:Intro}

Many knots in the 3-sphere have complement admitting a hyperbolic structure \cite{thurston:bulletin}, and by Mostow--Prasad rigidity \cite{mostow, prasad} and the Gordon--Luecke theorem \cite{GordonLuecke}, the structure is a complete knot invariant. However, it is still not well understood what properties of the hyperbolic metric distinguish knot complements from other hyperbolic 3-manifolds.

In \cite{PurcellSouto}, properties of hyperbolic knot complements are investigated by studying geometric limits. If a manifold $M$ is a geometric limit of knot complements, then there exist hyperbolic knots whose geometric properties are very close to those of $M$.
In \cite{PurcellSouto}, it is shown that any one-ended hyperbolic 3-manifold with finitely generated fundamental group that embeds in $\SSS^3$ is a geometric limit of hyperbolic knot complements. The class of such 3-manifolds is surprisingly broad, and includes hyperbolic 3-space $\HH^3$ itself. However, in \cite{KentSouto}, compact submanifolds of $\SSS^3$ are presented whose interiors cannot be homeomorphic to any geometric limit of hyperbolic knot complements. Thus although the class of geometric limits of knot complements is large, there are geometric and topological restrictions on the manifolds that appear. These are not well understood.

This paper continues the investigation of limits of hyperbolic knot complements, particularly those converging to $\HH^3$.
Because $\HH^3$ is a geometric limit of knot complements, for any $R>0$ there exists a hyperbolic knot $K$ and a point $x\in \SSS^3-K$ with injectivity radius $\inj_x(\SSS^3-K)$ at least $R$. Since the injectivity radius of points in a tubular neighborhood of $K$ can be arbitrarily small, this implies that in a sequence $\SSS^3-K_n$ converging to $\HH^3$, the cusp $N(K_n)$ must be pushed further and further outside larger and larger balls in $\SSS^3-K_n$.

For $M$ any hyperbolic 3-manifold, recall that the \emph{$R$-thick part} of $M$ is defined to be $M^{\geq R} = \{x\in M \mid \inj_x(M)\geq R\}$, and its complement, denoted $M^{<R}$ is the \emph{$R$-thin part}. The result of \cite{PurcellSouto} implies that for any $R$, there exists a knot $K$ whose $R$-thick part is nonempty. However, that paper does not consider the size of the set of all such points.

Brock and Dunfield \cite{BrockDunfield} recently asked: For any $R$, does there exist a sequence of knots for which the proportion of the volume of the $R$-thin part tends to zero? That is, for fixed $R$, it is known that there are knots for which the $R$-thick part is non-empty. But does there exist a sequence of knots for which the $R$-thick part takes up larger and larger proportions of the volume? This is called \emph{Benjamini--Schramm convergence to $\HH^3$}.

In this paper, we answer Brock and Dunfield's question in the affirmative.

\begin{theorem}\label{Thm:BSCknots}
Given $R>0$ and $\epsilon>0$, there is a knot $K\subset \SSS^3$ such that $M=\SSS^3-K$ is hyperbolic and
\begin{equation}\label{VolumeRatio}
  \frac{\vol(\{x\in M \mid \inj_x(M)\geq R\})}{\vol(M)}  =
  \frac{\vol(M^{\geq R})}{\vol(M)} > 1-\epsilon.
\end{equation}
\end{theorem}

The quantity $\vol(M^{\geq R})/\vol(M)$ will be called the \emph{$R$-volume ratio}.

The proof of \refthm{BSCknots} is a modification of Brock and Dunfield's \cite{BrockDunfield} proof of the existence of closed integral homology spheres Benjamini--Schramm converging to $\HH^3$, with added ingredients from \cite{KentPAMS} and \cite{Kent:Skinning}. Performing $1/n$-surgeries on the knots of \refthm{BSCknots} produces new examples of homology spheres Benjamini--Schramm converging to $\HH^3$.

\subsection{Organization and proof outline}

The paper is organized as follows. In \refsec{MappingTorus}, we build a mapping torus whose $R$-volume ratio satisfies the result of \refthm{BSCknots}. However, there is no guarantee that the mapping torus will be a knot in $\SSS^3$. To obtain a knot in $\SSS^3$, we will put the mapping torus into a topological construction, which is described in \refsec{Topological}. The construction consists of decomposing $\SSS^3$ into five pieces, two of which will be replaced by product manifolds coming from the mapping torus of \refsec{MappingTorus}. For these two pieces, we may take the $R$-volume ratio to be as large as we like, and we may take the products to have total volume as large as we like. The geometry on the remaining three of the five pieces is described in \refsec{GeometricTangles}. The geometry of these three pieces will be fixed once and for all. Thus by taking larger and larger product manifolds for the other two pieces, the proportion of volume coming from the three fixed pieces will become negligible. We describe all this in \refsec{GeometricKnots}, putting the pieces together to ensure the result is a knot in $\SSS^3$; this finishes the proof of \refthm{BSCknots}.

\subsection{Acknowledgements}
The first author was supported by an NSF CAREER grant DMS-1350075. 
The second author was supported by NSF grant DMS-1252687 and ARC grant DP160103085. 
Both authors were supported by von Neumann Fellowships at the Institute for Advanced Study: 
This material is based upon work supported by the National Science Foundation under agreement No. DMS-1128155. Any opinions, findings and conclusions or recommendations expressed in this material are those of the authors and do not necessarily reflect the views of the National Science Foundation.  The authors thank the referee for many useful comments that greatly improved the exposition.

\section{Thick mapping torus}\label{Sec:MappingTorus}

As in \cite{BrockDunfield}, we first build mapping tori satisfying \eqref{VolumeRatio}. 

Given a surface $S$ and a homeomorphism $f\from S\to S$, there is an associated mapping torus
\[ 
M_f = (S\times [-1,1]) / (x,-1) \sim (f(x), 1).
\]

The goal of this section is to prove the following. 

\begin{proposition}\label{Prop:ThickMappingTorus}
For any $R>0$ and any $\epsilon>0$, there is a closed surface $\Sigma$ of genus $g=g(R,\epsilon)$, an even number of points $\{p_1, \dots, p_{2k}\}\subset \Sigma$, and a pseudo-Anosov map $f$ on the punctured surface $S = \Sigma-\{p_1, \dots, p_{2k}\}$ such that the $R$-volume ratio of $M_f$ satisfies
\[ 
\frac{\vol(M_f^{\geq R})}{\vol(M_f)} > 1-\epsilon,
\]
and $f$ extends to a homeomorphism $\Sigma \to \Sigma$ taking each $p_i$ to itself. 
\end{proposition}

Recall that a \emph{cusp} $\mathbf{T}$ of a hyperbolic 3-manifold $M$ is a submanifold with a lift to the universal cover $\HH^3$ that is a horoball $H$. The subgroup of $\pi_1(M) \leq \PSL(2,\CC)$ fixing $H$ is either isometric to $\ZZ$, generated by a single parabolic element of $\PSL(2,\CC)$, or is isometric to $\ZZ\oplus\ZZ$, generated by two parabolic elements. In the first case, the cusp $\mathbf{T}\cong H/\ZZ$ is homeomorphic to the product of $\RR$ and an annulus; we call this a \emph{rank-1 cusp}. In the second case, the cusp $\mathbf{T}\cong H/\ZZ\oplus\ZZ$ is homeomorphic to the product of $\RR$ and a torus; we call this a \emph{rank-2 cusp}. In the rank-1 case, the boundary $\bdy\mathbf{T}$ is naturally a flat annulus; in the rank-2 case it is naturally a flat torus. In either case, if the quotient of a horoball embeds in a neighborhood of the cusp, we say the quotient of the horoball is a \emph{horoball neighborhood} of the cusp. We also use this notation and terminology for embedded disjoint unions of such neighborhoods. 
Finally, notice that for any $R>0$, any cusp will have a subset lying in the $R$-thin part. 

Now, $M_f$ obtained in \refprop{ThickMappingTorus} has rank-2 cusps corresponding to
\[ \{p_1,\dots,p_{2k}\}\times[-1,1]/\sim,\]
and so $\vol(M_f^{<R})>0$. To make the $R$-volume ratio large, we will ensure that the $R$-thin part lies only in the cusps and takes up a small proportion of the cusp volume.

\begin{lemma}\label{Lem:Cover}
For any $Q>0$, there is a surface $S$ of genus $g = g(Q)$ with an even number of punctures, and a pseudo-Anosov map $\varphi \from S \to S$ such that the $Q$-thin part $M_\varphi^{<Q}$ of $M_\varphi$ consists only of disjointly embedded horoball neighborhoods of the cusps of $M_\varphi$ corresponding to the punctures of $S$, and $\varphi$ fixes each puncture of $S$.
\end{lemma}

\begin{proof}
Let $f$ be any pseudo-Anosov homeomorphism on a punctured surface $\Sigma$ with more than one puncture. The mapping torus $M_f$ is hyperbolic \cite{Otal}, and the number of closed geodesics in $M_f$ of length at most $2Q$ is finite. These correspond to conjugacy classes $[\gamma_i]$ of elements of $\pi_1(M_f)$. Since $\pi_1(M_f)$ is residually finite \cite{malcev}, there is a finite-index normal subgroup of $\pi_1(M_f)$ that does not contain any $\gamma_i$. Let $N$ be the corresponding finite cover of $M_f$. Its shortest geodesic has length at least $2Q$. Moreover, the fibration of $M_f$ over $\mathbb{S}^1$ lifts to a fibration of $N$. 
 
Let $\mathbf{T}$ be a horoball neighborhood of the cusps of $N$. Let $[\delta_j]$ be the conjugacy classes of elements of $\pi_1(N)$ corresponding to the Euclidean geodesic loops on $\partial \mathbf{T}$ whose length is less than $2Q$.  Residual finiteness gives us a finite-index normal subgroup of $\pi_1(N)$ that contains none of the $\delta_j$, and passing to the corresponding finite cover of $N$ produces a fibered manifold $N'$ whose $Q$-thin part consists only of disjointly embedded horoball neighborhoods of the cusps.

The monodromy of the fibration $F\to N' \to \SSS^1$ is a homeomorphism of the punctured surface $F$ that permutes the punctures of $F$, so a finite power takes each puncture to itself. Let $N''$ be the corresponding finite cover of $N'$, so the monodromy $\mu$ of the fibration $F\to N''\to \SSS^1$ extends to a homeomorphism fixing all the punctures of $F$. 

It may be the case that the fiber of $N''$ has an odd number of punctures (even if $\Sigma$ were evenly punctured).
If so, we pass to a further cover as follows.
Let $\alpha$ and $\beta$ be elements of $H_1(N'';\ZZ)$ corresponding to distinct punctures of $F$. Let $\theta \from H_1(N'';\ZZ) \to \ZZ/2\ZZ$ be the homomorphism taking $\alpha$ and $\beta$ to $1$ and killing the other generators of $H_1(N'')$. 
Let $N'''$ be the $2$--fold cover of $N''$ corresponding to $\ker(\theta)$.
By construction of $\theta$, the lifts of $\alpha$ and $\beta$ are connected, and the lift of any other curve representing a different puncture has two components.
It follows that the fiber $S$ of $N'''$ has an even number of punctures. 
The monodromy $\varphi$ of the fibration $S \to N''' \to \mathbb{S}^1$ fixes the punctures of $S$ by construction, and so $N''' = M_\varphi$ is the desired manifold.
\end{proof}

\begin{proof}[Proof of \refprop{ThickMappingTorus}]
Let $R>0$ and $1 > \epsilon > 0$.
Recall that the function $\sinh(x)$ is increasing and unbounded. Thus for fixed $R$ and fixed $\epsilon$, we may choose $Q>R$ large enough that $\sinh(R/2)/\sinh(Q/2) < \epsilon$. 

Now, \reflem{Cover}, applied to $Q$, gives a mapping torus $M_\varphi$ with $Q$-thin part a horoball neighborhood of the cusps.
Since $Q>R$, the thin parts satisfy $M_\varphi^{<R} \subset M_\varphi^{<Q}$ and so the $R$-thin part is also a horoball neighborhood of the cusps. 

Let $\mathbf{T}_Q$ be a component of $\overline{M_\varphi^{<Q}}$, with $\mathbf{T}_R\subset \mathbf{T}_Q$ a component of $\overline{M_\varphi^{<R}}$.
We first prove that we have chosen $Q$ such that
\[ \frac{\vol(\mathbf{T}_R)}{\vol(\mathbf{T}_Q)} < \epsilon^2. \]
This can be seen as follows. 

Consider the universal cover $\mathbf{B}_Q$ of $\mathbf{T}_Q$ as a horoball centered at infinity in the upper half-space model; similarly $\mathbf{B}_R\subset \mathbf{B}_Q$ is a horoball. Thus in coordinates, the boundary
$\bdy\mathbf{B}_R = \{(x,y,h_R)\mid x,y\in\RR\}$ is a horosphere with height (i.e.\ third coordinate in $\HH^3$) some constant $h_R>0$. Similarly, $\bdy\mathbf{B}_Q$ is a horosphere with height $h_Q>0$. Because $\mathbf{T}_R\subset\mathbf{T}_Q$, it follows that the height $h_Q < h_R$.
Also, because $\bdy\mathbf{T}_Q$ and $\bdy\mathbf{T}_R$ have the same Euclidean shape (ignoring scale), the ratio
\[ \frac{\vol(\mathbf{T}_R)}{\vol(\mathbf{T}_Q)} = \frac{\area(\bdy\mathbf{T}_R)}{\area(\bdy\mathbf{T}_Q)} = \frac{h^2_Q}{h_R^2}. \]

Let $\Delta$ be the $\ZZ \oplus \ZZ$ group of parabolics such that $\mathbf{B}_Q / \Delta = \mathbf{T}_Q$. The shortest hyperbolic translation distance of $\Delta$ acting on $\partial \mathbf{B}_R$ is $R$, and on $\bdy\mathbf{B}_Q$ is $Q$. A calculation shows that the shortest Euclidean translation distance of $\Delta$ acting on $\bdy \mathbf{B}_R$ is then $2\sinh(R/2)$, and acting on $\bdy\mathbf{B}_Q$ is $2\sinh(Q/2)$. On the other hand, in coordinates the element of $\Delta$ translating the shortest distance takes some point $(x,y,h_R)$ to $(x',y',h_R)$; thus its Euclidean distance is \[1/h_R\cdot \sqrt{(x-x')^2 + (y-y')^2} = 2\sinh(R/2).\]
Similarly it takes $(x,y,h_Q)$ to $(x',y',h_Q)$, and so
\[1/h_Q\cdot \sqrt{(x-x')^2 + (y-y')^2} = 2\sinh(Q/2).\]
It follows that $h_Q/h_R = \sinh(R/2)/\sinh(Q/2)$. 

Thus, by our choice of $Q$, we have
\[ \frac{\vol(\mathbf{T}_R)}{\vol(\mathbf{T}_Q)} = \frac{h^2_Q}{h_R^2} =
\frac{\sinh^2(R/2)}{\sinh^2(Q/2)} < \epsilon^2. \]

Since this holds for each component of $M_\varphi^{<R}$, choosing $Q$ in this way guarantees that 
\[
\frac{\vol(M_\varphi^{<R})}{\vol(M_\varphi^{<Q})} < \epsilon^2.
\]

Then
\[ \frac{\vol(M_\varphi^{\geq R})}{\vol(M_\varphi)} = 1 - \frac{\vol(M_\varphi^{<R})}{\vol(M_\varphi^{\geq Q} + M_\varphi^{<Q})} >
1-\frac{\vol(M_\varphi^{<R})}{\vol(M_\varphi^{<Q})} > 1-\epsilon^2. \qedhere
\]
\end{proof}

\section{The topological construction}\label{Sec:Topological}

Our knot complements will be built from several geometric pieces.

Let $R, \epsilon > 0$ and let $g(R,\epsilon)$ be as in \refprop{ThickMappingTorus}.
Consider the Heegaard splitting of $\SSS^3$ of genus $g(R,\epsilon)$ with Heegaard surface $\Sigma$. So 
	\[
		\SSS^3 \cong (H_1' \cup_\Sigma H_2')/ \psi,
	\]
with $H_1'$, $H_2'$ handlebodies, $\bdy H_i' \cong \Sigma$, and $\psi \from \Sigma \to \Sigma$ a homeomorphism.

Consider a regular neighborhood of $\Sigma$ in $\SSS^3$ homeomorphic to $\Sigma\times[-1,1]$ made of two pieces $M_1'' \cong \Sigma\times[-1,0]\subset H_1'$ and $M_2''\cong \Sigma\times[0,1]\subset H_2'$ glued together by $\psi \from \Sigma\times\{0\}\to \Sigma\times\{0\}$. 

Let $H_1 = H_1' - M_1''$ and $H_2 = H_2'-M_2''$. We then obtain $\SSS^3$ by gluing four pieces: $H_1$, $M_1''$, $M_2''$, and $H_2$. Denote the gluing map between $H_1$ and $M_1''$ by $i_1\from \bdy H_1 \to \Sigma\times\{-1\}$, and that between $M_2''$ and $H_2$ by $i_2\from \Sigma\times\{1\}\to\bdy H_2$.

We further split each of $M_1''$ and $M_2''$ into two pieces, as follows. Let $M_1'$ be the manifold $\Sigma\times[-1, -\half]$ and let $M_1=\Sigma\times[-\half,0]$, with gluing map $i_1'\from \Sigma\times\{-\half\} \to \Sigma\times\{-\half\}$, and let $M_2=\Sigma\times[0,\half]$ and $M_2'=\Sigma\times[\half,1]$, with gluing map $i_2'\from \Sigma\times\{\half\}\to \Sigma\times\{\half\}$.

We let $M = (M_1 \cup M_2) / \psi.$

This is all illustrated schematically in \reffig{Chunks}.

\begin{figure}
  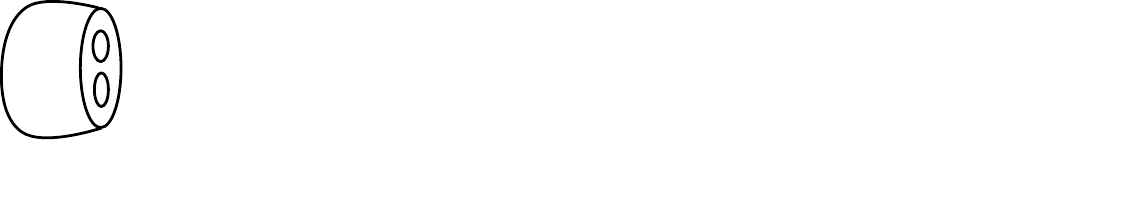
  \caption{Gluing all pieces yields $\SSS^3$}
  \label{Fig:Chunks}
\end{figure}

\section{Geometric tangles}\label{Sec:GeometricTangles}

We build our knots by gluing together tangles in $H_1$, $M_1'$, $M$, $M_2'$, and $H_2$. The pieces $H_1$, $M$, $H_2$ and their tangles are chosen first and are then fixed throughout the construction.

Start with the Heegaard splitting of $\SSS^3$ from the previous section. 
Let $K$ be any closed curve in $\SSS^3$ meeting the Heegaard surface $\Sigma$ transversely in $2k$ points $\{p_1,\dots,p_{2k}\}$. Recall that $M_1'\cup M\cup M_2'$ is homeomorphic to $\Sigma\times[-1,1]$, a regular neighborhood of $\Sigma$ in $\SSS^3$. 
We homotope $K$ so that it meets this neighborhood in arcs of the form $p_j\times[-1,1]$. 
We will homotope $K$ relative its intersection with $\partial H_1$, $\partial M_1'$, $\partial M$, $\partial M_2'$, and $\partial H_2$ so that its intersections with each of $H_1$, $M_1'$, $M$, $M_2'$, and $H_2$ are of a particularly nice form.


The setup is as follows. Let $N$ be a compact orientable $3$-manifold. (In our applications, $N$ will correspond to $H_1$, $H_2$, or $M$, but a more general result holds.) Let $A$ be a $1$-manifold properly embedded in $N$ such that each component of $\bdy N - \bdy A$ has negative Euler characteristic. That is, $A$ is a collection of embedded simple closed curves and arcs in $N$, with each component of $\bdy A$ lying on the closed surface $\bdy N$. Thus the surface $\bdy N-\bdy A$ is a punctured surface, with punctures corresponding to points $\bdy A$, and we require that the punctured surface have negative Euler characteristic. Finally, let $P$ be a pants decomposition of the punctured surface $\bdy N - \bdy A$. 

\begin{lemma}\label{Lem:ExtendPAMS}
Fix $\delta>0$. Let $N$, $A$, and $P$ be as above. 
Then there is a $1$-manifold $B$ homotopic to $A$ relative to its endpoints such that
  \begin{enumerate}
  \item $N-B$ is hyperbolic with totally geodesic boundary,
  \item the sum of lengths of geodesic representatives of all curves of $P$ on the geodesic boundary of $N-B$ is less than $\delta$, and
  \item all arcs in $B$ correspond to rank-$1$ cusps.
  \end{enumerate}
  Moreover, the manifold $N_{B,P} = N - (B \cup P)$ admits a hyperbolic structure with totally geodesic boundary a collection of $3$-punctured spheres.
\end{lemma}

The proof is nearly identical to that of \cite[Theorem~1]{KentPAMS}. Indeed, if we restrict to the case that $N$ has boundary a single closed surface and $A$ is a single simple closed curve in $N$, then \cite[Theorem~1]{KentPAMS} implies that for any simple closed curve $C$ on $N$, there exists a curve $B$ homotopic to $A$ such that $N-B$ is hyperbolic, $\bdy(N-B)$ is totally geodesic, and $C$ is isotopic on $\bdy(N-B)$ to a geodesic with length less than $\delta$. This implies items (1) and (2) for this restricted case, with $C$ replacing our $P$. To prove the full result of the lemma, we extend the proof of \cite[Theorem~1]{KentPAMS} to allow more general compact manifolds $N$, to allow $A$ to have multiple components including arc components, and to require $P$ to be a complete pants decomposition of $\bdy N-\bdy A$. However, the proof proceeds almost exactly as in \cite{KentPAMS}; we step through it briefly.

\begin{proof}[Proof of Lemma~\ref{Lem:ExtendPAMS}]
First, choose a surface $F$ properly embedded in $N$ with boundary containing $P$; for example, in our case we may take $F$ to be isotopic to $(\bdy N-\bdy A)-P$, pushed in a bit from the boundary. We will cut $N$ along $F$ and reglue via an element $\varphi$ carefully chosen in the mapping class group of $F$. We will denote the new manifold obtained by applying this regluing by $N(\varphi)$. 

A theorem of Myers \cite[Theorem~1.1]{Myers} implies that for any $N-F \cong N(\varphi)-F$ containing embedded arcs and curves $A-F$, there exists a $1$-manifold $k''$ homotopic to $A-F$ relative to its boundary such that each component of $(N-F)-k''$ is boundary incompressible, acylindrical, and atoroidal. Reglue $N-F$ along the traces of $F$ via $\varphi$, and let $K$ be the knot in $N(\varphi)$ corresponding to $k''$. The exterior $N(\varphi)-K$ is irreducible, atoroidal, and acylindrical by \cite[Lemma~2.1]{Myers}, so its double $d(N(\varphi)-K)$ admits a hyperbolic structure by Thurston's uniformization theorem for Haken manifolds (\cite{Morgan, thurston:bulletin}). By the proof of the Mostow--Prasad rigidity theorem, $\bdy(N(\varphi)-K)$ is a totally geodesic surface in $d(N(\varphi)-K)$, since it is fixed pointwise by an involution of that double. Thus cutting the double open again yields a hyperbolic structure on $N(\varphi)-K$ with totally geodesic boundary. Note that arcs of $K$ lie in the parabolic locus of $N(\varphi)-K$, and the boundary of a neighborhood of such an arc is homeomorphic to an annulus. Thus each arc of $K$ corresponds to a rank-$1$ cusp. Thus, provided we choose $\varphi$ such that $N(\varphi)$ is homeomorphic to $N$ and the image of $K$ under the homeomorphism is homotopic to $A$, these results imply items (1) and (3) of the lemma. (For example, this will hold if $\varphi$ is the identity map $\id$, but choosing $\varphi=\id$ will not allow us to prove item (2) in general.)

To finish the proof of the lemma, we choose the map $\varphi$ and show that our choice also implies item (2).

We use maps first studied by Birman \cite{Birman}, namely elements in the intersection of all Birman kernels for the punctures of $F$. (See also, for example, \cite[Chapter~4]{FarbMargalit}.) More carefully, let $F_K = F-K$, and let $\widehat{F}_{K,i}$ be the surface obtained by attaching a disk to the $i$-th boundary component of $F_K$. Let $\PMod(F_K)$ denote the pure mapping class group of $F_K$, i.e.\ the subgroup of the mapping class group which does not permute boundary components. Then we have the Birman exact sequence
\[ 1 \to \pi_1(\widehat{F}_{K,i}) \to \PMod(F_K) \overset{\phi_i}{\longrightarrow} \PMod(\widehat{F}_{K,i})\to 1. \]
Take $\psi$ to be a nontrivial element of $\bigcap_i \ker{\phi_i}$; the existence of a nontrivial $\psi$ is shown in \cite{KentPAMS}.
Note first that the map $\psi$ is Brunnian: whenever any boundary component of $F_K$ is filled, the resulting mapping class is the identity (see, for example, \cite{Whittlesey}). It follows that $N(\psi)$ is homeomorphic to $N$, and that $K$ in $N(\psi)$ is homotopic to $A$ in $N$; indeed, this must hold for any power $\psi^n$ of $\psi$. Finally, it is also shown in \cite{KentPAMS} that $\psi$ is psuedo-Anosov acting on $F$.

Now consider the manifolds $N^n = N(\psi^n)$. By the above work, for each $n$ there exists a $1$-manifold $B_n$ such that $N^n_K := N^n-B_n$ is hyperbolic with totally geodesic boundary, $N^n$ is homeomorphic to $N$, and the image of $B_n$ under this homeomorphism is homotopic to $A$.

On the other hand, let $d(F)=S$ in the doubled manifold $d(N^n_K)$; we abuse notation slightly and denote the double of $\psi$ by $\psi$. In the hyperbolic structure on $d(N^n_K)$, the subgroup $\pi_1(S)\leq \PSL(2,\CC)$ is quasi-Fuchsian, denoted by $Q(\psi^{-n}a_n, b_n)$ for some $a_n$, $b_n$ in the Teichm\"uller space $\mathcal{T}(S)$ of $S$. Thurston's bounded image theorem (see, for example \cite{Morgan}) implies that $Q(\psi^{-n}a_n,b_n)$ lies in a compact subset $A$ of $\mathcal{T}(S)\times\mathcal{T}(S)$, with $A$ independent of $n$.

Let $(a,b)\in A \subset \mathcal{T}(S)\times\mathcal{T}(S)$ lie in the same compact set. Since $\psi$ is reducible on $S$, fixing curves $P$, but pseudo-Anosov on the two components of $S-\bdy F$, a theorem of Brock \cite[Theorem~4.5]{Brock} implies that the lengths of the geodesic representatives of $P$ tend to zero in $Q(\psi^{n}a,b)$ as $n$ tends to infinity. But by \cite[Lemma~1]{KentPAMS}, $Q(\psi^{-n}a_n,b_n)$ and $Q(\psi^{n}a,b)$ are quasi-isometric independent of $n$, so the lengths of $P$ also tend to zero in $N^n_K$. This finishes the proof of parts (1)--(3).


For the last part of the lemma, note that the complement of a disjoint union of simple geodesics in a hyperbolic $3$-manifold admits a complete hyperbolic structure \cite{Kojima}. 
So the manifold obtained by removing $P$ from the double $d(N-B)$ of $N-B$ admits a complete hyperbolic structure.
Moreover, this manifold is the double $d( N_{B,P})$ of $N_{B,P}$, and it follows that $N_{B,P}$ admits a hyperbolic structure with totally geodesic boundary a collection of $3$-punctured spheres.
\end{proof}

If $\alpha$ is a curve on the boundary $T$ of a cusp neighborhood, then the normalized length of $\alpha$ is defined to be
	\[ 
		\calL(\alpha) = \frac{\len(\alpha)}{\sqrt{\area(T)}}. 
	\]

\begin{lemma}\label{Lem:ShortPantsBigMeridians}
Let $L>0$.  Then there is $\delta = \delta(\chi(F),L) > 0$ such that the following holds.
Let $N$ be an orientable finite-volume $3$-manifold with totally geodesic boundary in which the length of a pants decomposition $P$ is less than $\delta$. 
Then $N' = N - P$ admits a finite-volume hyperbolic metric with totally geodesic boundary a collection of $3$-punctured spheres, and, whenever $N'$ is embedded isometrically into a complete hyperbolic manifold $N''$ with $\bdy N'' = \emptyset$, the normalized length of each meridian corresponding to a curve of $P$ in $N''$ is at least $L$. 
\end{lemma}
\begin{proof}
Let $P_0$ be a curve in $P$ and let $\mathbf{T}_0$ be the corresponding cusp neighborhood in $N''$.
The \textit{meridian corresponding to $P_0$} is the filling slope on $\partial \mathbf{T}_0$ whose homology class dies under the projection homomorphism $H_1(\mathbf{T_0}) \cong \ZZ \oplus \ZZ \to \ZZ = \langle [P_0] \rangle$.

Let $\mathcal{A}$ be the horospherical annulus $\partial \mathbf{T}_0 \cap N'$.
It is shown in the proof of \cite[Theorem~39]{Kent:Skinning} that there is $\delta = \delta(\chi(F),L) > 0$ such that the conformal modulus of $\mathcal{A}$ is at least $L$ provided that the length of $P$ is less than $\delta$ in the totally geodesic boundary of $N$.
This implies that the normalized length of the meridian at $P_0$ is at least $L$.
\end{proof}

\section{Geometric knots}\label{Sec:GeometricKnots}

Our knot in $\SSS^3$ will be constructed by gluing together tangles in $H_1$, $M_1'$, $M$, $M_2'$, and $H_2$. 
The tangles in $M_1'$ and $M_2'$ will be braids provided by \refprop{ThickMappingTorus}.

For fixed $R>0$ and $\epsilon>0$, \refprop{ThickMappingTorus} provides a mapping torus $M_f$ whose $R$-volume ratio is at least $1-\epsilon$. Take the surface $\Sigma$, points $\{p_1, \dots, p_{2k}\}$,  $S=\Sigma-\{p_1,\dots,p_{2k}\}$, and $f$ as in that proposition. Let $M_f^\infty$ denote the infinite cyclic cover of $M_f$ corresponding to the fiber, and let $M_f^n$ denote the mapping torus corresponding to $f^n$. Finally, let $P$ be any pants decomposition of $S$. 

Given natural numbers $m$ and $n$, there exists a maximally cusped hyperbolic structure on $S\times\RR$ with curves $f^{-m}(P)$ corresponding to rank-$1$ cusps on one end, and $f^n(P)$ corresponding to rank-$1$ cusps on the other, by Thurston's Geometrization Theorem \cite{ThurstonII,Morgan,Otal}.
Let $S(f^{-m}(P), f^n(P))$ be this hyperbolic manifold.

The following lemma is almost identical to \cite[Proposition~3.4]{BrockDunfield}, only we allow additional cusps corresponding to the punctures in our surface. 
The lemma may be proven just as in that paper, using the fact that the Drilling Theorem applies to manifolds with additional cusps.

\begin{lemma}\label{Lem:MaxCuspsConverge}
The maximally cusped structures $S(f^{-m}(P), f^n(P))$ for $m,n>0$ converge strongly to the infinite cover $M_f^\infty$ as $m,n\to \infty$. 
The manifolds $S(P,f^n(P))$ converge strongly to a manifold $S_A$ with a degenerate end asymptotically isometric to the positive end of $M_f^\infty$ and whose convex core boundary is a surface with parabolic locus $P$. The analogous statement holds for $S(f^{-n}(P), P)$. \qed
\end{lemma}

We now prove \refthm{BSCknots}.

\begin{proof}[Proof of \refthm{BSCknots}]
For fixed $R$ and $\epsilon$, take $\Sigma$, $\{p_1, \dots, p_{2k}\}$, $S$, and $f$ as in \refprop{ThickMappingTorus}.
Take a Heegaard splitting of $\SSS^3$ with Heegaard surface $\Sigma$, and $H_1$, $M_1'$, $M$, $M_2'$, and $H_2$ as in \refsec{Topological}. 
Finally, take $K\subset\SSS^3$ as in \refsec{GeometricTangles}, meeting the Heegaard surface $\Sigma$ for $\SSS^3$ transversely in points $\{p_1, \dots, p_{2k}\}$, and let $P$ be any pants decomposition of $S$. 
Let $\delta > 0$.

Using \reflem{ExtendPAMS}, we homotope $K$ relative its intersection with $\partial H_1$ and 
$\partial H_2$ so that $H_i - K$ has a hyperbolic structure with totally geodesic boundary in which the length of $P$ is less than $\delta$.
Then $\overline{H}_i = H_i-(K\cup P)$ has a hyperbolic structure with totally geodesic boundary a collection of $3$-punctured spheres, for each $i$.

Now we consider $M_1' \cong S\times\big[-1,-\half\big]$ and $M_2'\cong S\times\big[\half,1\big]$.
For any integer $r>0$, the manifold $S\times \big[-1,-\half\big] - (P \times\{-1\} \cup f^r(P) \times \{-\half\})$ has hyperbolic structure with totally geodesic boundary, isometric to the convex core of $S(P,f^{r}(P))$.
Similarly, $S\times \big[\half,1\big] - (f^r(P) \times\{\half\} \cup P \times \{1\})$ has hyperbolic structure with totally geodesic boundary, isometric to the convex core of $S(f^{r}(P), P)$.

Consider $M=\Sigma\times\big[\!-\half,\half\big]$.
Let $A$ be the intersection of $K$ and $M$.
Let $B$ be the tangle homotopic to $A$ making the lengths of $P \times\{-\half\}$ and $P \times \{\half\}$ less than $\delta$ in $M - B$, given by \reflem{ExtendPAMS}.  
The manifold $M_{B,P} = M - (B \cup P \times \{\pm\half\})$ admits a hyperbolic structure with totally geodesic boundary a collection of $3$-punctured spheres.
The map $f^r$ extends to a self homeomorphism of $M = \Sigma\times\big[\!-\half,\half\big]$.
This homeomorphism takes $B$ to a collection of embedded arcs $f^r(B)$ in $f^r(M)$, and takes $P\times\{\pm\half\}$ to $f^r(P)\times\{\pm\half\}$. 

Glue  $\overline{H}_1$ and $\core(S(P,f^r(P)))$ together via $i_1$. 
Note this map takes $P$ to $P$ and takes $3$-punctured spheres to $3$-punctured spheres. Because there is a unique hyperbolic structure on a $3$-punctured sphere, this extends to an isometry.
Similarly, $i_1'$ induces an isometry from the right side of $\core(S(P,f^r(P)))$ to the left side of $f^r(M)$, $i_2'$ induces an isometry from the right side of $f^r(M)$ to the left side of $\core(S(f^r(P), P))$, and $i_2$ induces an isometry from the right side of $\core(S(f^r(P), P))$ to $\overline{H}_2$.

With these gluing maps, let
\[
N_r = \overline{H}_1 \cup \core(S(P,f^{r}(P))) \cup f^r(M_{B,P}) \cup \core(S(f^r(P),P)) \cup \overline{H}_2.
\]

By construction, $N_r$ is homeomorphic to $\SSS^3$ with a link removed. 
One component of this link is $K$, and the other components are copies of the components of $P$ and $f^r(P)$.
We call the cusps corresponding to the latter curves the \emph{horizontal} cusps.

The hyperbolic structures on $\overline{H}_1$, $\overline{H}_2$, and $M$ are fixed, and so are their volumes.
By Mostow--Prasad rigidity, the manifold $f^r(M)$ is isometric to $M$. 
It follows that the $R$-volume ratios of the links $N_r$ approach the $R$-volume ratio of  $M_f$ as $r$ goes to infinity. 
In particular, we may take $r$ such that the ratio is at least $1-\epsilon/2$.

To finish the proof, we we fill the horizontal cusps of our links to obtain the desired knots.
To do this, we use the Universal Hyperbolic Dehn Filling Theorem of Hodgson and Kerckhoff \cite{hk:univ} and the Drilling Theorem of Brock and Bromberg \cite{BrockBromberg}.
Together, these theorems provide a universal $L$ such that if the normalized length of each component of a Dehn filling slope is at least $L$, then the $\epsilon_3$-thick part of the filled manifold is $(1+\epsilon/100)$-bilipschitz to that of the original manifold, where $\epsilon_3$ is the $3$-dimensional Margulis constant.

In our case, we would like to fill our link complement along all of the horizontal meridians.  
By \reflem{ShortPantsBigMeridians}, there is $\delta$ such that the normalized lengths of these meridians will be at least $L$ provided the length of $P$ is less than $\delta$, and so we choose $\delta$ in this way.

Thus the $\epsilon_3$-thick part of the knot complement obtained by filling along the horizontal meridians is $(1+ \epsilon/100)$-bilipschitz to the $\epsilon_3$-thick part of our link complement.
Since the volume of the thin parts decrease under filling, the resulting knot is the desired one.
\end{proof}

\bibliographystyle{amsplain}
\bibliography{references}

\end{document}

%% file: figures/chunks.pdf_tex
\begingroup%
  \makeatletter%
  \providecommand\color[2][]{%
    \errmessage{(Inkscape) Color is used for the text in Inkscape, but the package 'color.sty' is not loaded}%
    \renewcommand\color[2][]{}%
  }%
  \providecommand\transparent[1]{%
    \errmessage{(Inkscape) Transparency is used (non-zero) for the text in Inkscape, but the package 'transparent.sty' is not loaded}%
    \renewcommand\transparent[1]{}%
  }%
  \providecommand\rotatebox[2]{#2}%
  \ifx\svgwidth\undefined%
    \setlength{\unitlength}{323.61365891bp}%
    \ifx\svgscale\undefined%
      \relax%
    \else%
      \setlength{\unitlength}{\unitlength * \real{\svgscale}}%
    \fi%
  \else%
    \setlength{\unitlength}{\svgwidth}%
  \fi%
  \global\let\svgwidth\undefined%
  \global\let\svgscale\undefined%
  \makeatother%
  \begin{picture}(1,0.18699429)%
    \put(0,0){\includegraphics[width=\unitlength,page=1]{chunks.pdf}}%
    \put(0.0103911,0.11576138){\color[rgb]{0,0,0}\makebox(0,0)[lb]{\smash{$H_1$}}}%
    \put(0.11350825,0.12800652){\color[rgb]{0,0,0}\makebox(0,0)[lb]{\smash{$i_1$}}}%
    \put(0,0){\includegraphics[width=\unitlength,page=2]{chunks.pdf}}%
    \put(0.20010154,0.11202579){\color[rgb]{0,0,0}\makebox(0,0)[lb]{\smash{$M_1'$}}}%
    \put(0.29971782,0.13580994){\color[rgb]{0,0,0}\makebox(0,0)[lb]{\smash{$i_1'$}}}%
    \put(0,0){\includegraphics[width=\unitlength,page=3]{chunks.pdf}}%
    \put(0.38131672,0.11046475){\color[rgb]{0,0,0}\makebox(0,0)[lb]{\smash{$M_1$}}}%
    \put(0.49140938,0.12894261){\color[rgb]{0,0,0}\makebox(0,0)[lb]{\smash{$\psi$}}}%
    \put(0,0){\includegraphics[width=\unitlength,page=4]{chunks.pdf}}%
    \put(0.5711354,0.10859182){\color[rgb]{0,0,0}\makebox(0,0)[lb]{\smash{$M_2$}}}%
    \put(0,0){\includegraphics[width=\unitlength,page=5]{chunks.pdf}}%
    \put(0.67814731,0.12519691){\color[rgb]{0,0,0}\makebox(0,0)[lb]{\smash{$i_2'$}}}%
    \put(0,0){\includegraphics[width=\unitlength,page=6]{chunks.pdf}}%
    \put(0.75787329,0.10484613){\color[rgb]{0,0,0}\makebox(0,0)[lb]{\smash{$M_2'$}}}%
    \put(0,0){\includegraphics[width=\unitlength,page=7]{chunks.pdf}}%
    \put(0.85687289,0.12777431){\color[rgb]{0,0,0}\makebox(0,0)[lb]{\smash{$i_2$}}}%
    \put(0.93927985,0.10467548){\color[rgb]{0,0,0}\makebox(0,0)[lb]{\smash{$H_2$}}}%
    \put(0,0){\includegraphics[width=\unitlength,page=8]{chunks.pdf}}%
    \put(0.48442409,0.00544691){\color[rgb]{0,0,0}\makebox(0,0)[lb]{\smash{$M$}}}%
  \end{picture}%
\endgroup%